\documentclass[a4paper,11pt]{article}
\usepackage{CJK}
\usepackage{amsmath,amssymb,graphicx}
\usepackage{subfigure}
\usepackage{lscape}
\usepackage{mathrsfs}
\usepackage{algorithmic}
\usepackage{multirow}
\usepackage{verbatim}
\usepackage{color}
\usepackage{ntheorem}
\usepackage{mathtools}
\usepackage{multirow}
\usepackage{enumerate}
\usepackage{amsfonts}
\usepackage{epsfig}
\usepackage{amscd}
\usepackage{stmaryrd}
\usepackage{arydshln}
\usepackage{youngtab}
\usepackage{young}
\usepackage{geometry}
\usepackage{bbm}
\usepackage{longtable}
\usepackage{url}
\usepackage{float}
\usepackage[colorlinks,
    linkcolor=blue,
    anchorcolor=blue,
    citecolor=blue,
    urlcolor=black,
    ]{hyperref}

\newtheorem{theorem}{Theorem}
\newtheorem{lemma}{Lemma}

\newtheorem{corollary}{Corollary}
\newtheorem{definition}{Definition}

\newenvironment*{proof}{\paragraph{Proof.}}{\hfill $\square$}
\newenvironment*{remark}{\paragraph{Remark.}}{\\}

\title{Quantum Algorithm for Matrix Logarithm by Integral Formula}
\author{Songling Zhang \footnotemark[2] \and Hua Xiang\footnotemark[2] \footnotemark[3] \footnotemark[4]}

\begin{document}
\date{}
\maketitle
\renewcommand{\thefootnote}{\fnsymbol{footnote}}
\footnotetext[2]{School of Mathematics and Statistics, Wuhan University, Wuhan 430072, China.}
\footnotetext[3]{Hubei Key Laboratory of Computational Science, Wuhan University, Wuhan 430072, China.}
\footnotetext[4]{Corresponding author. E-mail address: {\tt hxiang@whu.edu.cn}.}

\begin{abstract}
The matrix logarithm is one of the important matrix functions. Recently, a quantum algorithm that computes the state $|f\rangle$ corresponding to matrix-vector product $f(A)b$ is proposed in [Takahira, et al. Quantum algorithm for
matrix functions by Cauchy's integral formula, QIC, Vol.20, No.1\&2, pp.14-36, 2020]. However, it can not be applied to matrix logarithm. In this paper, we propose a quantum algorithm, which uses LCU method and block-encoding technique as subroutines, to compute the state $|f\rangle = \log(A)|b\rangle / \|\log(A)|b\rangle\|$ corresponding to $\log(A)b$ via the integral representation of $\log(A)$ and the Gauss-Legendre quadrature rule.
\end{abstract}

\section{Introduction}
Many quantum algorithms have been proposed and exhibit great advantages over their counterparts. For example, Harrow, Hassidim, and Lloyd proposed a well-known quantum algorithm, called HHL algorithm \cite{HHL}, for the quantum linear systems problem (QLSP). However, HHL algorithm has complexity $ \text{poly}(1/\epsilon)$ because of the use of phase estimation, where $ \epsilon $ is the precision parameter. Childs, Kothari and Somma used the method of linear combination of unitaries (LCU) to circumvent the limitations of phase estimation and obtain an improved algorithm \cite{LCU}, which exponentially improves the dependence on the parameter $\epsilon$, with complexity $ \text{poly}(\log (1/\epsilon))$ for QLSP. The LCU method is widely used for Hamiltonian simulation. As a simple example, consider implementing the operator $ A = \alpha_0 U_0 + \alpha_1 U_1 $, where $U_0$ and $U_1$ are unitaries, and $\alpha_i > 0 \ (i = 0, 1)$ without loss of generality. Consider a unitary $G_\alpha$ such that
$$ |0\rangle \mapsto \frac{1}{\sqrt{\alpha}} \left( \sqrt{\alpha_0}|0\rangle + \sqrt{\alpha_1}|1\rangle \right)
\ \text{and} \ \ |1\rangle \mapsto \frac{1}{\sqrt{\alpha}} \left( \sqrt{\alpha_0}|1\rangle - \sqrt{\alpha_1}|0\rangle \right), $$
where $\alpha = \alpha_0 + \alpha_1$. To implement $A$ on a state $|\psi\rangle$, let $U=|0\rangle \langle 0| \otimes U_0 + |1\rangle \langle 1| \otimes U_1$, then
\begin{equation*}
    \begin{aligned}
    |0\rangle |\psi\rangle & \xrightarrow{\ G_\alpha \otimes I \ } \frac{1}{\sqrt{\alpha}} \left(\sqrt{\alpha_0} |0\rangle + \sqrt{\alpha_1}|1\rangle \right)|\psi\rangle \\
    & \xrightarrow{\quad U \quad } \frac{1}{\sqrt{\alpha}} \left(\sqrt{\alpha_0} |0\rangle U_0 |\psi\rangle + \sqrt{\alpha_1} |1\rangle U_1 |\psi\rangle \right) \\
    & \xrightarrow{\ G^\dagger_\alpha \otimes I \ } \frac{1}{\alpha} \left( |0\rangle (\alpha_0 U_0 + \alpha_1 U_1) |\psi\rangle + \sqrt{\alpha_0 \alpha_1} |1\rangle (U_1 - U_0) |\psi\rangle \right).
    \end{aligned}
\end{equation*}
Measure the ancilla qubit and get the result 0, then we will have the desired state $A|\psi\rangle / \|A|\psi\rangle\| $. Furthermore, LCU method can be also used for implementing matrix functions $f(A)$. Recently, Takahira et al \cite{tosu20} proposed a quantum algorithm, which makes use of the contour integral of the matrix function $f(A)$, to compute the state $|f\rangle = f(A)|b\rangle / \| f(A)|b\rangle \|$ using HHL algorithm or LCU method as a subroutine. However, this quantum algorithm can be only applied when the contour is a circle centered at the origin, and cannot be applied to a function that has a singular at the origin, such as matrix logarithm \cite{tosu21}.

Matrix functions play an important role in many applications \cite{matrix-function}. For example, matrix exponential is related to the problem of Gibbs' state preparation \cite{gibbs}. Tong et al \cite{f-exp} recently constructed the block-encoding form of the matrix exponential $e^{-\beta H}$ for a positive semi-definite Hamiltonian $H \in \mathbb{C}^{N \times N}$ with $N=2^n$, using contour integral of exponential function and Gauss-Legendre quadrature. The contour integral representation for exponential function reads
$$ e^{-\beta x}=\frac{1}{2\pi \text{i}}\oint_\Gamma \frac{e^{-\beta z}}{z-x} dz, $$
where $x \geq 0$, and then we can express the matrix exponential as
$$ e^{-\beta H} = \frac{1}{2\pi \text{i}}\oint_\Gamma e^{-\beta z} (z-H)^{-1} dz, $$
where the contour they chosen is $ \Gamma = \{t^2 - a + \text{i}t \in \mathbb{C}: t \in \mathbb{R}\} $, where $a = 2b(1-b)$, $b=\min(1/2\beta, 1/6)$. Then we get the approximation as the form
$$ e^{-\beta H} \approx \sum_{j \in [J]} \omega_j(z_j I - H)^{-1} $$
by truncating the contour on a finite interval $t \in [-T, T]$, and applying the Gauss-Legendre quadrature formula \cite{f-exp}. Then the block-encoding of the approximation is given as
$$ \sum_{j \in [J]} \omega_j(z_j I - H)^{-1} = (\langle c| \otimes I_n) \left( \sum_{j \in [J]} |j\rangle \langle j| \otimes (z_j - H)^{-1} \right) (|c'\rangle \otimes I_n), $$
where $|c\rangle = \sum_j \sqrt{|\omega_j|} |j\rangle / \sum_j |\omega_j|$, $|c'\rangle = \sum_j \sqrt{|\omega_j|} e^{i\theta_j} |j\rangle / \sum_j |\omega_j|$, and $\theta_j$ satisfies $\omega_j = |\omega_j|e^{i\theta_j}$. Furthermore, Takahira et al \cite{tosu21} considered quantum algorithms based on the block-encoding framework for more general matrix functions.

In this work, we consider a quantum algorithm to compute the state $ |f\rangle = \log(A) |b\rangle / \| \log(A) |b\rangle \| $ corresponding to $\log(A) b$. Any non-singular square matrix $A$ has a logarithm; that is, there exists a matrix $X$ such that $e^X = A$. According to Theorem 1.31 of \cite{matrix-function}, when $A$ have no eigenvalues on $\mathbb{R}^-$, there is a unique logarithm $X$ of $A$ whose eigenvalues all lie in the strip $\{z: -\pi < \text{Im}(z) < \pi\}$. This is the principal logarithm denoted by $\log(A)$. In this paper, we are only interested in the principal logarithm. More properties of the matrix logarithm and classical implementations are given in \cite{matrix-function, Wouk, gl-pade, pade-error}.

In the remainder of this section, we give some necessary notations and definitions.
A matrix is $d$-sparse if there are at most $d$ non-zero entries in any row or column. For a $d$-sparse matrix $A \in \mathbb{C}^{N\times N}$, we assume that there exits an oracle $\mathcal{P}_A$ consisting of oracle $O_A$ and oracle $O_\nu$. Specifically, the oracle $O_A$ is the unitary operator that returns the matrix entries for given positions, that is,
\begin{equation} \label{O-A}
O_A |i,j,z\rangle = |i,j,z\oplus A_{ij}\rangle,
\end{equation}
for $i,j \in \{1,2,\cdots,N\}$, where $A_{ij}$ is the binary representation of the $(i,j)$ entry of $A$. And the oracle $O_\nu$ is the unitary operator that returns the locations of non-zero entries, that is,
\begin{equation} \label{O-v}
    O_\nu |j,l\rangle = |j,\nu(j,l)\rangle,
\end{equation}
for $j \in \{1,2,\cdots,N\}$ and $l \in \{1,2,\cdots,d\}$, where $\nu(j,l)$ is a function that returns the row index of the $l$-th non-zero entry of the $j$-th column. For simplicity, we assume that $N = 2^n$, where $n$ is a positive integer. For a given vector $b = (b_1, b_2, \cdots, b_N)^T \in \mathbb{C}^N$, we assume that there is an oracle $\mathcal{P}_b$ that prepares the state
$$ |b\rangle := \frac{\sum_{i=1}^{N} b_i |i\rangle}{\|\sum_{i=1}^{N} b_i |i\rangle\|} $$
corresponding to $b$ in time $O(\log N)$, that is,
\begin{equation} \label{Pb}
    \mathcal{P}_b |0^n\rangle = |b\rangle.
\end{equation}
In this paper, $\| \cdot \|$ represents 2-norm $\| \cdot \|_2$ for a vector or an operator.

Formally, we define the QLSP \cite{LCU} as follows.

\begin{definition} \label{qlsp}
(Quantum Linear Systems Problem)
Let $A \in \mathbb{C}^{N \times N}$ be a $d$-sparse Hermitian matrix such that $\|A\| \leq 1$. Let $b \in \mathbb{C}^N$. Suppose that there exits an oracle $\mathcal{P}_A$ consisting of $O_A$ and $O_\nu$ in Eq.(\ref{O-A}) and (\ref{O-v}), and an oracle $\mathcal{P}_b$ in Eq.(\ref{Pb}). For linear system $Ax = b$, we define the state $|x\rangle$ as
$$ |x\rangle := \frac{\sum_{i=1}^{N} x_i |i\rangle}{\| \sum_{i=1}^{N} x_i |i\rangle \|}, $$
where $x_i$ is the $i$-th entry of $x$. Then, the goal of the problem is to output a state $|\tilde{x}\rangle$ such that $\| |x\rangle - |\tilde{x}\rangle \| \leq \epsilon$ with a probability of at least $1/2$, where $0 \leq \epsilon \leq 1/2$.
\end{definition}

If $A$ is not Hermitian, consider the following linear system
\begin{equation}
    \begin{bmatrix}
        0 &A \\
        A^\dagger &0 \\
    \end{bmatrix}
    \begin{bmatrix}
        0 \\ x \\
    \end{bmatrix}
    =
    \begin{bmatrix}
        b \\ 0 \\
    \end{bmatrix}
\end{equation}
instead of the original linear system $Ax=b$. So we can always assume that the matrix $A$ is Hermitian without loss of generality.

For the problem in Definition \ref{qlsp}, there are two well-known quantum algorithms, i.e., HHL algorithm \cite{HHL} and LCU method \cite{LCU}, which can be used to solve it. In this paper, we use LCU method as subroutine. The following theorem about LCU method is well-known.
\begin{theorem} \label{HHL-LCU}
    (\cite{LCU}, Theorem 4)
    The problem defined in Definition \ref{qlsp} can be solved using
    $$ O\left( d \kappa_A^2 \log^2 \left(\frac{d \kappa_A}{\epsilon}\right) \right) $$
    queries to oracle $\mathcal{P}_A$ and
    $$ O\left( \kappa_A \log\left( \frac{d \kappa_A}{\epsilon}\right) \right) $$
    uses of  $\ \mathcal{P}_b$, with gate complexity
    $$ O\left( d \kappa_A^2 \log^2\left( \frac{d \kappa_A}{\epsilon} \right) \left[ \log N + \log^{\frac{5}{2}} \left( \frac{d \kappa_A}{\epsilon}\right) \right] \right), $$
    where $\kappa_A = \|A\| \|A^{-1}\|$ is the condition number of matrix $A$.
\end{theorem}

In Section 2, we will present the approximation of the matrix logarithm by its integral formula and analyze the approximation error. In Section 3, we will give the framework of the quantum algorithm. In Section 4 and 5, we will analyze the proposed quantum algorithm in error estimation and success probability respectively. In Section 6, we will consider how to apply the LCU method to the new linear system related to the original one, and estimate the total complexity of this algorithm. Finally, we will draw conclusions in Section 7.

\section{Approximation by integral formula}
In this section, we will introduce the integral representation of the matrix logarithm given by Wouk \cite{Wouk}, and the approximation using the Gauss-Legendre quadrature.

We first give the definition of principle logarithm as follows.

\begin{definition}
    (\cite{matrix-function}, Theorem 1.31)
    Let $A \in \mathbb{C}^{N \times N}$ have no eigenvalues on $ \mathbb{R}^- $. There is a unique logarithm $X$ of $A$, whose eigenvalues all lie in the strip $\{z: -\pi < \text{Im}(z) < \pi\}$. We refer to $X$ as the principal logarithm of A and write $X = \log(A)$.
\end{definition}

According to Wouk \cite{Wouk}, we have the integral representation of the matrix logarithm $\log(A)$ in the following.

\begin{theorem} \label{integral-representation}
(\cite{matrix-function}, Theorem 11.1) For $A \in \mathbb{C}^{N \times N}$ with no eigenvalues on $\mathbb{R}^-$,
\begin{equation} \label{logA}
\log(A) = \int_{0}^{1} (A-I)[t(A-I)+I]^{-1} dt.
\end{equation}
\end{theorem}

For simplicity, we define
\begin{equation}
    h(t) := [t(A-I)+I]^{-1},
\end{equation}
then $ \log(A) = (A-I)\int_{0}^{1} h(t) dt $.
Next, we apply the Gauss-Legendre quadrature formula to discretize this integral. This leads to
\begin{equation}
    \begin{aligned}
    \log(A) &= (A-I)\int_{0}^{1} h(t) dt \\
    & \simeq (A-I) \cdot \frac{1}{2} \sum_{j=1}^{M} s_j h\left(\frac{1}{2} t_j + \frac{1}{2}\right) \\
    & = \sum_{j=1}^{M} \omega_j (A-I) h(\tau_j) \\
    & := f_M(A),
    \end{aligned}
\end{equation}
where $ \tau_j = \frac{1}{2} t_j + \frac{1}{2}, \ \omega_j = \frac{1}{2} s_j, $
and $t_j, s_j$ are the nodes and weights of $M$-point Gauss-Legendre quadrature respectively. In the above approximation, we define
\begin{equation} \label{logMA}
    f_M(A) = (A-I)\sum_{j=1}^{M} \omega_j [\tau_j (A-I) + I]^{-1}.
\end{equation}

In the following, we bound the error of this approximation. Under the assumption $\|A-I\| < 1$, Lemma \ref{error-pade} gives the error bound of matrix logarithm and its Pad\'{e} approximation. Lemma \ref{pade-gl} shows the relation of $M$-point Gauss-Legendre quadrature $f_M(A)$ and the Pad\'{e} approximation to $\log(A)$.

\begin{lemma} (\cite{pade-error}) \label{error-pade}
    Let $r_m(x)$ be the $(m,m)$ Pad\'{e} approximation to $\log(1+x)$, the numerator and denominator of which are polynomials in $x$ of degree $m$. For a given matrix $X \in \mathbb{C}^{N \times N}$ with $\|X\|<1$, we have the bound
    $$ \|r_m(X) - \log(I+X)\| \leq | r_m(-\|X\|) - \log(1-\|X\|)|. $$
\end{lemma}

\begin{lemma} (\cite{gl-pade}, Theorem 4.3) \label{pade-gl}
    Let $\| A-I \| < 1$. Then $ f_M(A) $, the $M$-point Gauss-Legendre quadrature rule for $\log(A)$, is the $(M,M)$ Pad\'{e} approximation to $\log(A)$.
\end{lemma}

From Lemma \ref{error-pade} and Lemma \ref{pade-gl}, we obtain the error bound of $\log(A)$ in Eq.(\ref{logA}) and its approximation $f_M(A)$ in Eq.(\ref{logMA}).

\begin{theorem} \label{th-3}
    Let $ \|A-I\| < 1$, then for matrix logarithm $\log(A)$ in Eq.(\ref{logA}) and approximation $f_M(A)$ in Eq.(\ref{logMA}), we have a bound
    \begin{equation} \label{fM-error}
    \|f_M(A)-\log(A)\| \leq |r_M(-\|A-I\|) - \log(1-\|A-I\|)|:=\epsilon_{M,A},
    \end{equation}
    where $r_M(x)$ is the $(M,M)$ Pad\'{e} approximation to $\log(1+x)$.
\end{theorem}

\begin{proof}
    Let $X = A-I$. From Lemma \ref{pade-gl}, we have that $f_M(A)$ is the $(M,M)$ diagonal Pad\'{e} approximation $r_M(X)$, that is, $r_M(X) = f_M(A)$. Then using Lemma \ref{error-pade}, it's easy to obtain Eq.(\ref{fM-error}).
\end{proof}

\section{Quantum algorithm} \label{algorithm}
In this section, we describe the framework of the quantum algorithm to compute state $|f\rangle$, using the above $M$-point Gauss-Legendre quadrature $f_M(A)$.

The original goal is to output state
\begin{equation}
    |f\rangle = \frac{\log(A) |b\rangle}{\| \log(A)|b\rangle \|}.
\end{equation}
Now, using the approximation $ f_M(A) $, we consider constructing a quantum algorithm that outputs state
\begin{equation}
    |f_M\rangle:=\frac{f_M(A)|b\rangle}{\| f_M(A)|b\rangle \|}
\end{equation}
instead of $|f\rangle$.

We define
$$ h_M(A) := \sum_{j=1}^{M} \omega_j [\tau_j (A-I) + I]^{-1}, $$
then
$$ f_M(A)b = (A-I) h_M(A) b, $$
therefore we first consider the state
\begin{equation}
|h_M\rangle := \frac{h_M(A)|b\rangle}{\| h_M(A)|b\rangle \|}.
\end{equation}
As we can see in the following,
$$ h_M(A) b = \sum_{j=1}^{M} \omega_j [\tau_j (A-I) + I]^{-1} b = \sum_{j=1}^{M} \omega_j x^{(j)}, $$
where $ x^{(j)} = [\tau_j (A-I) + I]^{-1} b, \ j \in \{1,\cdots,M\} $.
Then, with the definitions of $h_M(A)$ and $|h_M\rangle$, we have
\begin{equation}
\begin{aligned}
    |f_M\rangle & = \frac{f_M(A)|b\rangle}{\| f_M(A)|b\rangle \|}
    = \frac{(A-I)h_M(A)|b\rangle}{\| (A-I)h_M(A)|b\rangle \|} \\
    & = \frac{\|h_M(A) |b\rangle \| \cdot (A-I) |h_M\rangle}{\|h_M(A) |b\rangle \| \cdot \|(A-I) |h_M\rangle\|} \\
    & = \frac{(A-I)|h_M\rangle}{\|(A-I) |h_M\rangle\|}.
\end{aligned}
\end{equation}

We provide a brief description of the quantum algorithm that outputs state $|f_M\rangle$. Details and analysis of the method are discussed in the next several sections.

\begin{enumerate} [\textbf{Step 1.}]
    \item
    Applying the HHL algorithm \cite{HHL} or LCU method \cite{LCU} to obtain quantum state $|\boldsymbol{x}\rangle$ corresponding to solution $\boldsymbol{x} \in \mathbb{C}^{NM}$ of linear system
    \begin{equation}
    \boldsymbol{A}\boldsymbol{x} = \boldsymbol{b},
    \end{equation}
    where $\boldsymbol{A} \in \mathbb{C}^{NM \times NM} $ is a block diagonal matrix defined as
    \begin{equation} \label{new-A}
    \boldsymbol{A} :=
    \begin{pmatrix}
    \tau_1 (A-I)+I & & & \\
    &\tau_2 (A-I)+I & & \\
    & &\ddots & \\
    & & &\tau_M (A-I)+I \\
    \end{pmatrix}
    \end{equation}
    and
    \begin{equation}
    \boldsymbol{x} :=
    \begin{pmatrix}
    x^{(1)} \\
    x^{(2)} \\
    \vdots \\
    x^{(M)} \\
    \end{pmatrix},\quad
    \boldsymbol{b} :=
    \begin{pmatrix}
    b\\
    b\\
    \vdots \\
    b\\
    \end{pmatrix}
    =
    \begin{pmatrix}
    1\\ 1\\ \vdots\\ 1\\
    \end{pmatrix} \otimes b.
    \end{equation}
    Then we have
    \begin{equation}
    |\boldsymbol{x}\rangle = \sum_{j=1}^{M}\frac{\|x^{(j)}\|}{\|\boldsymbol{x}\|} |j\rangle |x_j\rangle = \sum_{j=1}^{M} p_j |j\rangle |x_j\rangle,
    \end{equation}
    where $ |x_j\rangle = \sum_i x_i^{(j)} |i\rangle / \| \sum_i x_i^{(j)} |i\rangle \| $, $ p_j = \|x^{(j)}\| / \|\boldsymbol{x}\| $, and $ x_i^{(j)} $ stands for the $i$-th component of $ x^{(j)} $, $j \in \{1,\cdots, M\}$.
\end{enumerate}

\begin{enumerate}[\textbf{Step 2.}]
    \item
    Add an ancilla qubit and perform a controlled rotation on $|j\rangle$ such that
    \begin{equation}
    \begin{aligned}
    \sum_{j=1}^{M} p_j |j\rangle |x_j\rangle |0\rangle & \mapsto \sum_{j=1}^{M} p_j |j\rangle |x_j\rangle
    \left( C_M \omega_j |0\rangle + \sqrt{1-|C_M \omega_j|^2} |1\rangle \right) \\
    & = C_M \sum_{j=1}^{M} p_j \omega_j |j\rangle |x_j\rangle |0\rangle + |\Phi_1^\perp \rangle \\
    & := |\psi_1\rangle, \\
    \end{aligned}
    \end{equation}
    where $C_M$ is defined by
    \begin{equation}
    C_M := \frac{1}{\max_j \omega_j},
    \end{equation}
    and $|\Phi_1^\perp \rangle$ satisfies $(I_M \otimes I_N \otimes |0\rangle \langle 0|) |\Phi_1^\perp \rangle = 0$.
\end{enumerate}

\begin{enumerate}[\textbf{Step 3.}]
    \item
    For simplicity, let $M = 2^m$, where $m$ is a positive integer. Applying the Hadamard gates $H^{\otimes m} \otimes I_N \otimes I_2$ to $|\psi_1\rangle$, then we have
    \begin{equation} \label{step3-res}
    \begin{aligned}
    &\quad \ (H^{\otimes m} \otimes I_N \otimes I_2)
    \left( C_M \sum_{j=1}^{M} p_j \omega_j |j\rangle |x_j\rangle |0\rangle + |\Phi_1^\perp \rangle \right) \\
    & = \frac{C_M}{\sqrt{M}} \sum_{j=1}^{M} p_j \omega_j |0^m\rangle |x_j\rangle |0\rangle + |\Phi_2^\perp \rangle \\
    & = |0^m\rangle \otimes \frac{C_M}{\sqrt{M}} \sum_{j=1}^{M} \frac{\| x^{(j)} \|}{\| \boldsymbol{x}\|} \omega_j \sum_{i=1}^{N} \frac{x_i^{(j)}}{\| x^{(j)} \|} |i\rangle \otimes |0\rangle + |\Phi_2^\perp \rangle \\
    & = |0^m\rangle \otimes \frac{C_M}{\sqrt{M} \|\boldsymbol{x}\|} \sum_{i=1}^{N} \left( \sum_{j=1}^{M} \omega_j x_i^{(j)} \right)|i\rangle \otimes |0\rangle + |\Phi_2^\perp \rangle \\
    & = |0^m\rangle \otimes \frac{C_M \|h_M(A)b \|}{\sqrt{M} \|\boldsymbol{x}\|} |h_M\rangle \otimes |0\rangle + |\Phi_2^\perp \rangle \\
    & := |\psi_2\rangle, \\
    \end{aligned}
    \end{equation}
    where $|\Phi_2^\perp \rangle$ satisfies $(|0^m\rangle \langle 0^m| \otimes I_N \otimes |0\rangle \langle 0|)|\Phi_2^\perp \rangle = 0$.
\end{enumerate}

\begin{enumerate}[\textbf{Step 4.}]
    \item
    Suppose the block-encoding form of $A/\alpha$ as
    $$ U_A = \begin{pmatrix}
    A/\alpha & *  \\
    * & * \\
    \end{pmatrix} \in \mathbb{C}^{2^a N \times 2^a N}, $$
    which satisfies $ (\langle 0^a| \otimes I_N) U_A (|0^a\rangle \otimes I_N) = A/\alpha $, where $\alpha$ is a positive scalar that satisfies $ \|A/\alpha\| \leq 1 $, $ a \in \mathbb{N} $. Obviously, for all states $|\varphi\rangle \in \mathbb{C}^N$,
    $$ U_A |0^a\rangle |\varphi\rangle = \frac{1}{\alpha} |0^a\rangle A |\varphi\rangle + |\Psi_1^\perp\rangle,
    $$
    where $|\Psi_1^\perp\rangle$ satisfies $(|0^a\rangle \langle 0^a| \otimes I_N) |\Psi_1^\perp\rangle = 0$. Then we consider how to construct the block-encoding of $A-I$ based on $U_A$. Let
    $$ V = |0\rangle \langle 0| \otimes U_A - |1\rangle \langle 1| \otimes I_{2^a N}, $$
    then for all states $|\phi\rangle \in \mathbb{C}^{2^a N}$,
    \begin{equation}
    \begin{aligned}
        |0\rangle |\phi\rangle & \xmapsto{\ G \otimes I_{2^a N} \ } \frac{1}{\sqrt{1 + \alpha}} (\sqrt{\alpha}|0\rangle + |1\rangle) |\phi\rangle \\
        & \xmapsto{\quad  V \quad} \frac{1}{\sqrt{1 + \alpha}} (\sqrt{\alpha}|0\rangle U_A |\phi\rangle - |1\rangle |\phi\rangle) \\
        & \xmapsto{\ G^\dagger \otimes I_{2^a N} \ } \frac{1}{1 + \alpha} \left( |0\rangle (\alpha U_A - I_{2^a N}) |\phi\rangle - \sqrt{\alpha}|1\rangle (U_A + I_{2^a N}) |\phi\rangle \right),
    \end{aligned}
    \end{equation}
    where $G$ is the unitary operator such that
    $$ G|0\rangle = \frac{1}{\sqrt{1 + \alpha}} (\sqrt{\alpha} |0\rangle + |1\rangle),\quad G|1\rangle = \frac{1}{\sqrt{1 + \alpha}} (\sqrt{\alpha}|1\rangle - |0\rangle), $$
    that is,
    $$ G = \frac{1}{\sqrt{1 + \alpha}} \begin{pmatrix}
    \sqrt{\alpha} & -1  \\
    1 & \sqrt{\alpha} \\
    \end{pmatrix}. $$
    Let
    \begin{equation}
        U = (G^\dagger \otimes I_{2^a N}) V (G \otimes I_{2^a N}),
    \end{equation}
    and it is the block-encoding of $U_A - \frac{1}{\alpha} I_{2^a N}$. It's easy to prove that
    $$
    \begin{aligned}
        & \quad \ (\langle 0^a| \otimes I_N) (\langle 0| \otimes I_{2^a N}) U (|0\rangle \otimes I_{2^a N}) (|0^a\rangle \otimes I_N) \\
        & = (\langle 0^a| \otimes I_N) \frac{\alpha}{1 + \alpha} (U_A - \frac{1}{\alpha} I_{2^a N}) (|0^a\rangle \otimes I_N) \\
        & = \frac{1}{1 + \alpha}(A-I),
    \end{aligned}
    $$
    i.e., $(1 + \alpha) (\langle 0^{a+1}| \otimes I_N) U (|0^{a+1}\rangle \otimes I_N) =  A-I $.
    Consider the case $|\phi\rangle = |0^a\rangle |\varphi\rangle$ as
    follows.
    \begin{equation}
    \begin{aligned}
        U |0\rangle |\phi\rangle
        & = \frac{1}{1 + \alpha} \bigg( |0\rangle (\alpha U_A - I_{2^a N}) |\phi\rangle - \sqrt{\alpha}|1\rangle (U_A + I_{2^a N}) |\phi\rangle \bigg) \\
        & = \frac{1}{1 + \alpha} \bigg( |0\rangle (\alpha U_A - I_{2^a N}) |0^a\rangle |\varphi\rangle - \sqrt{\alpha}|1\rangle (U_A + I_{2^a N}) |0^a\rangle |\varphi\rangle \bigg) \\
        & = \frac{1}{1 + \alpha} |0\rangle \bigg(|0^a\rangle (A-I) |\varphi\rangle) + \alpha |\Psi_1^\perp\rangle \bigg) - \frac{\sqrt{\alpha}}{1 + \alpha} |1\rangle \left(|0^a\rangle (\frac{1}{\alpha} A + I) |\varphi\rangle + |\Psi_1^\perp\rangle \right) \\
        & = \frac{1}{1 + \alpha} |0^{a+1}\rangle(A-I) |\varphi\rangle + |\Psi^{\perp}\rangle,
    \end{aligned}
    \end{equation}
    where $|\Psi^{\perp}\rangle$ satisfies $(|0^{a+1}\rangle \langle 0^{a+1}| \otimes I_N) |\Psi^{\perp}\rangle = 0$.
    Then we have
    \begin{equation} \label{proba}
    \begin{aligned}
    & \quad \ (I_M \otimes U \otimes I_2) |0^{a+1}\rangle |\psi_2\rangle \\
    & = (I_M \otimes U \otimes I_2) |0^{a+1}\rangle \left( |0^m\rangle \otimes \frac{C_M \|h_M(A)b \|}{\sqrt{M} \|\boldsymbol{x}\|} |h_M\rangle \otimes |0\rangle + |\Phi_2^\perp \rangle \right) \\
    & = |0^m\rangle \otimes U \left( |0^{a+1}\rangle \otimes \frac{C_M \|h_M(A)b \|}{\sqrt{M} \|\boldsymbol{x}\|} |h_M\rangle \right) \otimes |0\rangle  + (I_M \otimes U \otimes I_2) |0^{a+1}\rangle |\Phi_2^\perp \rangle \\
    & = |0^{m+a+1}\rangle \otimes \frac{C_M \|h_M(A)b \|}{(1 + \alpha) \sqrt{M} \|\boldsymbol{x}\|} (A-I) |h_M\rangle \otimes |0\rangle + |\Phi^{\perp}\rangle \\
    & = |0^{m+a+1}\rangle \otimes \frac{C_M \|h_M(A)b \|}{(1 + \alpha) \sqrt{M} \|\boldsymbol{x}\|} \|(A-I) |h_M\rangle\| \frac{(A-I) |h_M\rangle}{\|(A-I) |h_M\rangle\|} \otimes |0\rangle + |\Phi^{\perp}\rangle \\
    & = |0^{m+a+1}\rangle \otimes \frac{C_M \|(A-I) h_M(A)b \|}{(1 + \alpha) \sqrt{M} \|\boldsymbol{x}\|}  |f_M\rangle \otimes |0\rangle + |\Phi^{\perp}\rangle,
    \end{aligned}
    \end{equation}
    where $|\Phi^{\perp}\rangle$ satisfies $(|0^{m+a+1}\rangle \langle 0^{m+a+1}| \otimes I_N \otimes |0\rangle \langle 0|) |\Phi^{\perp}\rangle = 0$.
\end{enumerate}

\begin{enumerate}[\textbf{Step 5.}]
    \item
    Measure the first register and the ancilla qubit. Conditioned on seeing $00 \cdots 0$ on the first $(m+a+1)$ qubits and $0$ on the ancilla qubit, we have the desired state $|f_M\rangle = (A-I)|h_M\rangle / \| (A-I)|h_M\rangle \|$.
\end{enumerate}

\section{Error estimates}
In order to analyze the error, we define the vectors that describe state $|\tilde{\boldsymbol{x}}\rangle$, $|f\rangle$, $|f_M\rangle$ and $|\tilde{f}\rangle$ in the following.

\begin{definition} \label{def-x}
    Let $\epsilon' > 0$. For state $|\tilde{\boldsymbol{x}}\rangle$ such that $\| |\boldsymbol{x}\rangle - |\tilde{\boldsymbol{x}}\rangle \| \leq \epsilon'$, define $NM$-dimensional vectors $\tilde{\boldsymbol{x}}$ such that $|\tilde{\boldsymbol{x}}\rangle = \sum_i \tilde{\boldsymbol{x}}_i |i\rangle / \| \sum_i \tilde{\boldsymbol{x}}_i |i\rangle \|$, where $\tilde{\boldsymbol{x}}_i$ is the $i$-th element of $\tilde{\boldsymbol{x}}$, and define vectors $\tilde{x}^{(k)} \in \mathbb{C}^M (k=0,1,\cdots,M-1)$ such that $\tilde{\boldsymbol{x}} = (\tilde{x}^{(0)}, \tilde{x}^{(1)}, \cdots, \tilde{x}^{(M-1)})^T$.
\end{definition}

\begin{definition} \label{f-fm}
    Define
    \begin{equation}
    \begin{aligned}
    \boldsymbol{f} & :=\frac{1}{\| \boldsymbol{x} \|} \log(A) b, \\
    \boldsymbol{f}_M & := \frac{1}{\| \boldsymbol{x} \|} f_M(A) b = \frac{1}{\| \boldsymbol{x} \|} (A-I) \sum_{j=1}^{M} \omega_j x^{(j)},\\
    \boldsymbol{\tilde{f}} & := \frac{1}{\|\tilde{\boldsymbol{x}}\|} (A-I) \sum_{j=1}^{M} \omega_j \tilde{x}^{(j)}\\
    \end{aligned}
    \end{equation}
    such that $|f\rangle = \sum_i f_i |i\rangle / \| \sum_i f_i |i\rangle \|$, $|f_M\rangle = \sum_i f_{M_i} |i\rangle / \| \sum_i f_{M_i} |i\rangle \|$, and $|\tilde{f}\rangle = \sum_i \tilde{f}_i |i\rangle / \| \sum_i \tilde{f}_i |i\rangle \|$ respectively, where $ f_i$, $ f_{M_i}$ and $ \tilde{f}_i $ is the i-th element of $ \boldsymbol{f} $, $ \boldsymbol{f}_M $ and $ \boldsymbol{\tilde{f}} $ respectively.
\end{definition}

Then in order to get the bound of $ \| |f\rangle - |\tilde{f}\rangle \| $, we first compute the bounds of $ \| \boldsymbol{f} - \boldsymbol{f}_M \| $ and $ \| \boldsymbol{f}_M - \boldsymbol{\tilde{f}} \| $, which is shown in Lemma \ref{l-7} and Lemma \ref{l-8}.

\begin{lemma} \label{l-7}
    With the assumption that $\|A-I\| < 1$, for vectors $\boldsymbol{f}$ and $\boldsymbol{f}_M$, the following inequality holds.
    \begin{equation}
    \| \boldsymbol{f} - \boldsymbol{f}_M \| \leq \frac{1}{\|\boldsymbol{x}\|} \| f(A) - f_M(A) \| \|b\|.
    \end{equation}
\end{lemma}

\begin{proof}
    Using the definitions of $\boldsymbol{f}$ and $\boldsymbol{f}_M$, we have this result. 
\end{proof}

\begin{lemma} \label{l-8}
    With the assumption that $\|A-I\| < 1$, for vector $\boldsymbol{f}_M$ and $\boldsymbol{\tilde{f}}$, the following equation holds.
    \begin{equation}
    \| \boldsymbol{f}_M - \boldsymbol{\tilde{f}} \| \leq \delta_M \epsilon',
    \end{equation}
    where $\delta_M = \|A-I\| \left( \sum_{j=1}^{M}|\omega_j|^2 \right)^{\frac{1}{2}}$, and $\epsilon'$ is defined in Definition \ref{def-x}.
\end{lemma}

\begin{proof}
    Using the definitions, we have
    \begin{equation}
    \begin{aligned}
    \| \boldsymbol{f}_M - \boldsymbol{\tilde{f}} \| & = \left\| \frac{1}{\| \boldsymbol{x} \|} (A-I) \sum_{j=1}^{M} \omega_j x^{(j)} - \frac{1}{\|\tilde{\boldsymbol{x}}\|} (A-I) \sum_{j=1}^{M} \omega_j \tilde{x}^{(j)} \right\| \\
    & \leq \|A-I\| \left\| \sum_{j=1}^{M} \omega_j \left( \frac{x^{(j)}}{\| \boldsymbol{x} \|} - \frac{\tilde{x}^{(j)}}{\|\tilde{\boldsymbol{x}}\|}\right) \right\|. \\
    \end{aligned}
    \end{equation}
    Then using Cauchy-Schwarz inequality, we can obtain
    \begin{equation}
    \begin{aligned}
    \| \boldsymbol{f}_M - \boldsymbol{\tilde{f}} \| & \leq \|A-I\| \left( \sum_{j=1}^{M}|\omega_j|^2 \right)^{\frac{1}{2}} \left( \sum_{j=1}^{M} \left\| \frac{x^{(j)}}{\| \boldsymbol{x} \|} - \frac{\tilde{x}^{(j)}}{\|\tilde{\boldsymbol{x}}\|} \right\|^2 \right)^{\frac{1}{2}} \\
    & = \|A-I\| \left( \sum_{j=1}^{M}|\omega_j|^2 \right)^{\frac{1}{2}} \left\| \frac{\boldsymbol{x}}{\| \boldsymbol{x} \|} - \frac{\tilde{\boldsymbol{x}}}{\|\tilde{\boldsymbol{x}}\|} \right\| \\
    & = \|A-I\| \left( \sum_{j=1}^{M}|\omega_j|^2 \right)^{\frac{1}{2}} \| |\boldsymbol{x}\rangle - |\tilde{\boldsymbol{x}}\rangle \| \\
    & < \|A-I\| \left( \sum_{j=1}^{M}|\omega_j|^2 \right)^{\frac{1}{2}} \epsilon', \\
    \end{aligned}
    \end{equation}
    completing the proof.
\end{proof}

Before computing the error bound, we present the following lemma, which will be used in the proof of the error bound.
\begin{lemma} \label{l-9}
    (\cite{tosu20}, Lemma 14)
    For any vectors $v \in \mathbb{C}^N$ and $w \in \mathbb{C}^N$, the following inequality holds.
    \begin{equation}
    \left\| \frac{v}{\|v\|} - \frac{w}{\|w\|} \right\| \leq 2\frac{\|v-w\|}{\|v\|}.
    \end{equation}
\end{lemma}

\begin{proof}
    Using the triangle inequality, we get
    \begin{equation}
    \begin{aligned}
    \left\| \frac{v}{\|v\|} - \frac{w}{\|w\|} \right\| & = \left\| \frac{v}{\|v\|} - \frac{w}{\|v\|} + \frac{w}{\|v\|} - \frac{w}{\|w\|} \right\| \\
    & \leq \left\| \frac{v}{\|v\|} - \frac{w}{\|v\|} \right\| + \left\| \frac{w}{\|v\|} - \frac{w}{\|w\|} \right\| \\
    & \leq \frac{\|v-w\|}{\|v\|} + \left| \frac{1}{\|v\|} - \frac{1}{\|w\|}\right| \|w\| \\
    & = \frac{\|v-w\|}{\|v\|} + \frac{| \|w\| - \|v\| |}{\|v\|}.
    \end{aligned}
    \end{equation}
    Again using the triangle inequality, we have $\|v\| = \|v-w+w\| \leq \|v-w\| + \|w\|$. Thus, $\left| \|v\|-\|w\| \right| \leq \|v-w\|$ and then $$ \left\| \frac{v}{\|v\|} - \frac{w}{\|w\|} \right\| \leq 2\frac{\|v-w\|}{\|v\|}. $$ 
\end{proof}

\begin{theorem} \label{final-error}
    Let $\|A-I\|<1$. We assume that the improved version of the HHL algorithm (i.e., LCU method) in Step 1 outputs state $|\tilde{\boldsymbol{x}}\rangle$ such that $\| |\boldsymbol{x}\rangle - |\tilde{\boldsymbol{x}}\rangle \| \leq \epsilon'$, where $0 < \epsilon' < \frac{1}{2}$. Then for the error of state $|\tilde{f}\rangle$, which is the actual output state of the algorithm, we have the bound
    \begin{equation} \label{def-epsilon}
    \left\| |f\rangle - |\tilde{f}\rangle \right\| \leq \frac{2}{K} \left(\epsilon_{M,A} + \sqrt{M} \delta_M  \epsilon' \right) \equiv \epsilon,
    \end{equation}
    where $K=\|f(A)|b\rangle\|(1-\|A-I\|)$.
\end{theorem}

\begin{proof} \label{output-error}
    Making use of Definition \ref{f-fm} and Lemma \ref{l-9}, we have
    $$ \left\| |f\rangle - |\tilde{f}\rangle \right\| = \left\| \frac{\boldsymbol{f}}{\|\boldsymbol{f}\|} - \frac{\boldsymbol{\tilde{f}}}{\| \boldsymbol{\tilde{f}} \|} \right\| \leq 2 \frac{\| \boldsymbol{f} - \boldsymbol{\tilde{f}} \|}{\| \boldsymbol{f} \|} = 2 \frac{\| \boldsymbol{x} \|}{\|f(A)b\|} \cdot \| \boldsymbol{f} - \boldsymbol{\tilde{f}} \|. $$
    Then, using the triangle inequality, we have
    \begin{equation}
    \begin{aligned}
    \left\| |f\rangle - |\tilde{f}\rangle \right\|
    & \leq 2 \frac{\| \boldsymbol{x} \|}{\|f(A)b\|} \left( \Big\| \boldsymbol{f} - \boldsymbol{f}_M \Big\| + \left\| \boldsymbol{f}_M - \boldsymbol{\tilde{f}} \right\| \right) \\
    & \leq 2 \frac{\| \boldsymbol{x} \|}{\|f(A)b\|} \left( \frac{1}{\|\boldsymbol{x}\|} \| f(A) - f_M(A) \| \|b\| + \delta_M \epsilon' \right) \\
    & = 2 \frac{1}{\|f(A)b\|} \bigg( \| f(A) - f_M(A) \| \|b\| + \| \boldsymbol{x} \| \delta_M \epsilon' \bigg) \\
    & = 2 \frac{1}{\|f(A)|b\rangle\|} \left( \| f(A) - f_M(A) \| + \frac{\| \boldsymbol{x} \|}{\| b \|} \delta_M \epsilon' \right) \\
    & \leq 2 \frac{1}{\|f(A)|b\rangle\|} \left( \epsilon_{M,A} + \frac{\| \boldsymbol{x} \|}{\| b \|} \delta_M \epsilon' \right), \\
    \end{aligned}
    \end{equation}
    where the last inequality uses Theorem \ref{th-3}.
    From $ \boldsymbol{A}\boldsymbol{x} = \boldsymbol{b} $, we can get that $ \| \boldsymbol{x} \| \leq \| \boldsymbol{A}^{-1} \| \| \boldsymbol{b} \| $, so that $$ \left\| |f\rangle - |\tilde{f}\rangle \right\| \leq 2 \frac{1}{\|f(A)|b\rangle\|} \left( \epsilon_{M,A} + \frac{\|\boldsymbol{A}^{-1}\| \|\boldsymbol{b}\|}{\| b \|} \delta_M \epsilon' \right). $$
    Since $ \boldsymbol{b} = (b, b, \cdots, b)^T$, then
    \begin{equation} \label{b-b'}
    \|\boldsymbol{b}\| = \sqrt{M} \|b\|.
    \end{equation}
    Using this equation (\ref{b-b'}) and the upper bound of $ \|\boldsymbol{A}^{-1}\| $ that will be shown in Lemma \ref{norm-A-1}, we obtain
    \begin{equation}
    \begin{aligned}
    \left\| |f\rangle - |\tilde{f}\rangle \right\|
    & \leq 2 \frac{1}{\|f(A)|b\rangle\|} \left( \epsilon_{M,A} + \frac{ \sqrt{M} \delta_M \epsilon'}{1-\|A-I\|} \right) \\
    & = \frac{2}{\|f(A)|b\rangle\|(1-\|A-I\|)} \left(\epsilon_{M,A}(1-\|A-I\|) + \sqrt{M} \delta_M \epsilon' \right) \\
    & \leq \frac{2}{\|f(A)|b\rangle\|(1-\|A-I\|)} \left(\epsilon_{M,A} + \sqrt{M} \delta_M \epsilon' \right) \\
    & = \frac{2}{K} \left(\epsilon_{M,A} + \sqrt{M} \delta_M \epsilon' \right). \\
    \end{aligned}
    \end{equation}
\end{proof}

\section{Success probability}
In this section, we show the lower bound of the success probability of the algorithm. According to (\ref{proba}), the ideal success probability is
\begin{equation}
p = \frac{C_M^2}{(1 + \alpha)^2 M \|\boldsymbol{x}\|^2}
\|(A-I)h_M(A)b\|^2 .
\end{equation}
It can be simplified as following,
\begin{equation} \label{ideal-p}
\begin{aligned}
p   = \frac{C_M^2}{(1 + \alpha)^2 M \|\boldsymbol{x}\|^2} \|f_M(A)b\|^2 = \frac{C_M^2}{(1 + \alpha)^2 M} \|\boldsymbol{f}_M\|^2 . \\
\end{aligned}
\end{equation}
According to Lemma (\ref{l-8}), $\|\boldsymbol{f}_M\| \leq  \|\boldsymbol{\tilde{f}}\| + \delta_M \epsilon'$, here $0 < \epsilon' < 1/2$. Usually, $\epsilon' \ll 1$. Therefore, the actual success probability can be described as
\begin{equation}
\tilde{p} = \frac{C_M^2}{(1 + \alpha)^2 M} \|\boldsymbol{\tilde{f}}\|^2.
\end{equation}

\begin{remark}
    A more accurate upper bound of (\ref{ideal-p}) reads
    \begin{equation*}
    \begin{aligned}
    p & \leq \frac{C_M^2}{(1 + \alpha)^2 M} \left( \|\boldsymbol{\tilde{f}}\| + \delta_M \epsilon' \right)^2 \leq \frac{C_M^2}{(1 + \alpha)^2 M} \left( \|\boldsymbol{\tilde{f}}\| + \frac{\sqrt{M}}{C_M} \epsilon' \right)^2 \\
    & \approx \frac{C_M^2}{(1 + \alpha)^2 M} \|\boldsymbol{\tilde{f}}\|^2 + \frac{2 C_M}{(1 + \alpha)^2 \sqrt{M}} \|\boldsymbol{\tilde{f}}\| \epsilon',
    \end{aligned}
    \end{equation*}
    where in the second inequality we use the bound $$ \delta_M \leq (M \max_j \omega_j^2)^\frac{1}{2} = \frac{\sqrt{M}}{C_M}. $$
    When $\frac{C_M}{2\sqrt{M}} \|\boldsymbol{\tilde{f}}\| \gg \epsilon'$, we can neglect the second term.
\end{remark}

\begin{theorem}
    Assuming that the LCU method in Step 1 outputs state $|\tilde{\boldsymbol{x}}\rangle$ such that $\| |\boldsymbol{x}\rangle - |\tilde{\boldsymbol{x}}\rangle \| \leq \epsilon'$, where $0 \leq \epsilon' \leq 1/2$. And by an appropriate setting of parameters $\epsilon'$ and $M$, we assume that $\| |f\rangle - |\tilde{f}\rangle \| \leq \epsilon$, where $\epsilon$ is defined in Eq.(\ref{def-epsilon}). Then
    \begin{equation}
    \tilde{p} \geq \left(1 - \frac{\epsilon}{2}\right)^2 \left( \frac{C_M K}{(1 + \alpha) M} \right)^2,
    \end{equation}
    where $K=\|f(A)|b\rangle\|(1-\|A-I\|)$.
\end{theorem}

\begin{proof}
    By Theorem \ref{final-error}, we can derive that $2\|\boldsymbol{f} - \boldsymbol{\tilde{f}}\| / \|\boldsymbol{f}\| \leq \epsilon$. Thus, using the triangle inequality, we have
    \begin{equation}
    \|\boldsymbol{f}\| \leq \|\boldsymbol{f}-\boldsymbol{\tilde{f}}\| + \|\boldsymbol{\tilde{f}}\| \leq \frac{\epsilon}{2}\|\boldsymbol{f}\| + \|\boldsymbol{\tilde{f}}\|.
    \end{equation}
    Therefore,
    \begin{equation}
    \begin{aligned}
    \|\boldsymbol{\tilde{f}}\| & \geq \left(1 - \frac{\epsilon}{2}\right) \|\boldsymbol{f}\| = \left(1 - \frac{\epsilon}{2}\right) \frac{\|f(A)b\|}{\|\boldsymbol{x}\|} \\
    & \geq \left(1 - \frac{\epsilon}{2}\right) \frac{\|f(A)|b\rangle\|}{\|\boldsymbol{A}^{-1}\| \sqrt{M}} \geq \left(1 - \frac{\epsilon}{2}\right) \frac{\|f(A)|b\rangle\|(1-\|A-I\|)}{\sqrt{M}}.
    \end{aligned}
    \end{equation}
    Thus, the lower bound of $\tilde{p}$ reads
    \begin{equation}
    \begin{aligned}
        \tilde{p} & = \frac{C_M^2}{(1 + \alpha)^2 M} \|\boldsymbol{\tilde{f}}\|^2 \\
        & \geq \frac{C_M^2}{(1 + \alpha)^2 M} \left( \left(1 - \frac{\epsilon}{2}\right) \frac{\|f(A)|b\rangle\|(1-\|A-I\|)}{\sqrt{M}} \right)^2 \\
        & = \left(1 - \frac{\epsilon}{2}\right)^2 \left( \frac{C_M K}{(1 + \alpha) M} \right)^2.
    \end{aligned}
    \end{equation}
\end{proof}

\begin{remark}
    It seems that the lower bound of $\tilde{p}$ may be very small when $M$ is large, but actually the parameter $C_M = \frac{1}{\max_j \omega_j} (1 \leq j \leq M)$ increases with the increase of $M$. Since $$ \frac{1}{M} = \frac{1}{M} \sum_{j=1}^{M} \omega_j \leq \max_j \omega_j \leq 1, $$
    then $ 1 \leq C_M \leq M$. Hence, $$\tilde{p} \geq \left(1 - \frac{\epsilon}{2}\right)^2  \frac{K^2}{(1 + \alpha)^2 M^2}. $$
\end{remark}

Furthermore, using the technique of amplitude amplification \cite{amp}, we can repeat Step 1 to Step 4 of the algorithm $ O(1/\sqrt{\tilde{p}}) $ times to obtained a constant success probability.

\section{Linear solver for $\boldsymbol{A}\boldsymbol{x} = \boldsymbol{b}$}
In this section, we discuss the detailed implementation of the Step 1 of the algorithm, and derive complexity estimate in Theorem \ref{th-complexity}.

\subsection{The condition number of $\boldsymbol{A}$}
In this subsection, we consider the condition number of matrix $\boldsymbol{A}$ in Eq.(\ref{new-A}), which will be used in the Theorem \ref{th-complexity}. We first present the upper bound of $\|\boldsymbol{A}\|$ and $\|\boldsymbol{A}^{-1}\|$.

\begin{lemma} \label{norm-A}
    Let $\|A-I\| < 1$, then for the matrix $\boldsymbol{A}$ in Eq.(\ref{new-A}), $\|\boldsymbol{A}\| < 2$ holds.
\end{lemma}

\begin{proof}
    $\boldsymbol{A}$ can be rewritten as $\boldsymbol{A} = \text{diag}(\tau_1,\tau_2,\cdots,\tau_M) \otimes (A-I) + I_M \otimes I_N$. For square matrices $X$ and $Y$, $\|X \otimes Y\| = \|X\| \|Y\|$, and note that $\tau_j \in (0,1)$ for $j \in \{1,2,\cdots,M\}$. Thus, we have
    \begin{equation}
    \begin{aligned}
    \|\boldsymbol{A}\| & \leq \| \text{diag}(\tau_1,\tau_2,\cdots,\tau_M) \otimes (A-I) \| + \| I_M \otimes I_N \| \\
    & = \| \text{diag}(\tau_1,\tau_2,\cdots,\tau_M) \| \cdot \| A-I \| + 1 \\
    & \leq \|A-I\| + 1 < 2 .\\
    \end{aligned}
    \end{equation} 
\end{proof}

\begin{lemma} \label{norm-A-1}
    Let $\|A-I\| < 1$, then for the matrix $\boldsymbol{A}$ in Eq.(\ref{new-A}), $\|\boldsymbol{A}^{-1}\| < \frac{1}{1-\|A-I\|}$ holds.
\end{lemma}

\begin{proof}
    For $j \in \{1,2,\cdots,M\}$,
    $$\|(\tau_j (A-I)+I)^{-1}\| \leq \frac{1}{1-\tau_j \|A-I\|} < \frac{1}{1-\|A-I\|}.$$
    It is obvious that $\boldsymbol{A}^{-1}$ is a block diagonal matrix with diagonal blocks $(\tau_j(A-I)+I)^{-1}$ $(j \in \{1,2,\cdots,M\})$. Thus,
    \begin{equation}
    \|\boldsymbol{A}^{-1}\| = \max_j \{ \|(\tau_j (A-I)+I)^{-1}\| \} < \frac{1}{1-\|A-I\|}.
    \end{equation}
\end{proof}

With Lemma \ref{norm-A} and Lemma \ref{norm-A-1}, we can easily get the upper bound of condition number of $\boldsymbol{A}$ in the following corollary.

\begin{corollary} \label{condition-num}
Let $\|A-I\| < 1$. Then the condition number $\kappa_{\boldsymbol{A}}:=\|\boldsymbol{A}\| \|\boldsymbol{A}^{-1}\|$ of $\boldsymbol{A}$ in Eq.(\ref{new-A}) is bounded as
\begin{equation}
\kappa_{\boldsymbol{A}} < \frac{2}{1-\|A-I\|}.
\end{equation}
\end{corollary}

\subsection{The Oracles $\mathcal{P}_{\boldsymbol{A}}$ and $\mathcal{P}_{\boldsymbol{b}}$}
Now, with the assumption that we can use the oracle $\mathcal{P}_A$ in (\ref{O-A}), we consider how to construct the oracle $\mathcal{P}_{\boldsymbol{A}}$, which consists of the oracle $\mathcal{O}_{\boldsymbol{A}}$ returning the entry of $\boldsymbol{A}$ and the oracle $\mathcal{O}_{\boldsymbol{\nu}}$ returning the positions of the nonzero entries of $\boldsymbol{A}$. For simplicity, we assume that all diagonal entries of $A$ are nonzero.

We follow the analysis of \cite{tosu20}. We first consider the construction of oracle $\mathcal{O}_{\boldsymbol{\nu}}$ using oracle $\mathcal{O}_{\nu}$ in (\ref{O-v}) and CNOT gates. Note that the row index of $l$-th nonzero element of $(kN+j)$-th column of $\boldsymbol{A}$ is $kN + v(j,l)$ for $k \in \{0,1,\cdots,M-1\} $ and $j \in \{1,2,\cdots,N\}$. Thus, we have
\begin{equation}
|k,j\rangle |0^m,l\rangle \
\xmapsto{\mathcal{O}_\nu}
|k,j\rangle |0^m,v(j,l)\rangle
\xmapsto{\text{CNOT}}
|k,j\rangle |k,v(j,l)\rangle
\equiv \mathcal{O}_{\boldsymbol{\nu}}|k,j\rangle |0^m,l\rangle.
\end{equation}
From it we can easily know that the query complexity and gate complexity of $\mathcal{O}_{\boldsymbol{\nu}}$ is $O(1)$ and $O(m) = O(\log M)$ respectively.

Next, we consider how to construct oracle $\mathcal{O}_{\boldsymbol{A}}$ using oracle $O_A$. Noting that the non-diagonal blocks of $\boldsymbol{A}$ are zero, so we can consider oracle $\mathcal{O}_{\boldsymbol{A}}$ that satisfies
\begin{equation}
\mathcal{O}_{\boldsymbol{A}}|k,i\rangle |k',j\rangle |0\rangle =
\begin{cases}
|k,i\rangle |k',j\rangle |0\rangle, \qquad \qquad \qquad \qquad \  \text{when} \ k \neq k', \\
|k,i\rangle |k,j\rangle |\tau_k A_{ij} + (1-\tau_k)\delta_{ij}\rangle, \quad \text{when} \ k = k'.
\end{cases}
\end{equation}

Using quantum registers $|0\rangle_{r_1}, |0\rangle_{r_2}, |0\rangle_{r_3}, |0\rangle_{r_4}$ and auxiliary flag qubits $|0\rangle_{a_1}, |0\rangle_{a_2}$, we can implement oracle $\mathcal{O}_{\boldsymbol{A}}$ by the following procedures.

\begin{itemize}
    \setlength{\itemindent}{4pt}
    \item[\uppercase\expandafter{\romannumeral1.}] Perform $|k,k'\rangle |0\rangle_{a_1} \mapsto |k,k'\rangle |\delta_{k,k'}\rangle_{a_1} $ to check whether $k$ and $k'$ are equal.

    \item[\uppercase\expandafter{\romannumeral2.}] Conditioned on the $a_1$ qubit being 1:
    \begin{itemize}
        \item[(1)] Perform $|i,j\rangle |0\rangle_{r_1} \mapsto |i,j\rangle |A_{ij}\rangle_{r_1}$ using oracle $\mathcal{O}_A$.
        \item[(2)] Perform $|k\rangle |A_{ij}\rangle_{r_1} |0\rangle_{r_2} \mapsto |k\rangle |A_{ij}\rangle_{r_1} |\tau_k A_{ij}\rangle_{r_2}$ using the multiplier.
        \item[(3)] Perform $|i,j\rangle |0\rangle_{a_2} \mapsto |i,j\rangle |\delta_{ij}\rangle_{a_2}$ to check whether $i$ and $j$ are equal.
        \item[(4)] If $a_2$ qubit is also 1, perform $|k\rangle|0\rangle_{r_3} \mapsto |k\rangle |1-\tau_k \rangle_{r_3}$ using the adder.
        \item[(5)] Perform $|\tau_k A_{ij}\rangle_{r_2} |v\rangle_{r_3} |0\rangle_{r_4} \mapsto |\tau_k A_{ij}\rangle_{r_2} |v\rangle_{r_3} |\tau_k A_{ij}+v\rangle_{r_4}$ using the adder, where $v=1-\tau_k$ when the $a_1$ qubit being 1, and $v=0$ when the $a_1$ qubit being 0.
    \end{itemize}
    \item[\uppercase\expandafter{\romannumeral3.}] Uncomputing the '$r_1$', '$r_2$', '$r_3$' registers and the '$a_1$', '$a_2$' qubits, the value stored in the '$r_4$' register represents $\tau_k (A_{ij} - \delta_{ij}) + \delta_{ij}$ for $k = k'$ and 0 for $k \neq k'$.
\end{itemize}

Assuming that $\mathcal{O}_{\boldsymbol{A}}$ returns the element of $\boldsymbol{A}$ with $s$ bits of accuracy, we consider the gate complexity of $\mathcal{O}_{\boldsymbol{A}}$, and show in the following table.

\begin{table}[H]
    \small
    \begin{center}
        \caption{Gate complexity of oracle $\mathcal{O}_{A'}$}
        \renewcommand\arraystretch{2} 
        \begin{tabular}{|c|c|c|c|}
            \hline
            \multicolumn{2}{|c|}{\textbf{Step}} & \textbf{Gate Complexity} & \textbf{Description} \\
            \hline
            \multicolumn{2}{|c|}{\uppercase\expandafter{\romannumeral1}} & $O(\log M)$ & $|k,k'\rangle |0\rangle_{a_1} \mapsto |k,k'\rangle |\delta_{k,k'}\rangle_{a_1} $ needs $O(\log M)$ gates \\
            \hline
            \multirow{5}{*}{\uppercase\expandafter{\romannumeral2}} & (1) & $0$ & No gate complexity, only use one query to oracle $\mathcal{O}_A$ \\
            \cline{2-4}
            \multirow{5}{*}{} & (2) & $O(s^2)$ & $O(s^2)$ gates to obtain the value of $\tau_k A_{ij}$ with $s$ bits of accuracy \\
            \cline{2-4}
            \multirow{5}{*}{} & (3) & $O(\log N)$ & $|i,j\rangle |0\rangle_{a_2} \mapsto |i,j\rangle |\delta_{ij}\rangle_{a_2}$ needs $O(\log N)$ gates \\
            \cline{2-4}
            \multirow{5}{*}{} & (4) & $O(s)$ & $O(s)$ gates to obtain the value of $1-\tau_k$ with $s$ bits of accuracy \\
            \cline{2-4}
            \multirow{5}{*}{} & (5) & $O(s)$ & The same reason as above \\
            \hline
            \multicolumn{2}{|c|}{\uppercase\expandafter{\romannumeral3}} & $O(\log (MN) + s^2)$ & The sum of step \uppercase\expandafter{\romannumeral1} and step \uppercase\expandafter{\romannumeral2} \\
            \hline
        \end{tabular}
    \end{center}
\end{table}

As shown in the Table 1, in step \uppercase\expandafter{\romannumeral1}, the gate complexity is $O(\log M)$ for
$$|k,k'\rangle |0\rangle_{a_1} \mapsto |k,k'\rangle |\delta_{k,k'}\rangle_{a_1}. $$
In step \uppercase\expandafter{\romannumeral2}-(1), it only uses one query to the oracle $\mathcal{O}_A$ and no gate complexity. For simplicity, we assume that the oracle $\mathcal{O}_A$ returns the element with $s$ bits of accuracy and $\tau_k$ $(k \in \{1,2,\cdots,M\})$ has $s$ bits. In step \uppercase\expandafter{\romannumeral2}-(2), it uses $O(s^2)$ gates to obtain the value of $\tau_k A_{ij}$ with $s$ bits of accuracy \cite{tosu20}. In step \uppercase\expandafter{\romannumeral2}-(3), the gate complexity for
$$|i,j\rangle |0\rangle_{a_2} \mapsto |i,j\rangle |\delta_{ij}\rangle_{a_2}$$
is $O(\log N)$. In step \uppercase\expandafter{\romannumeral2}-(4), we can obtain $1-\tau_k$ with $s$ bits of accuracy using $O(s)$ gates \cite{tosu20}. In step \uppercase\expandafter{\romannumeral2}-(5), it also uses $O(s)$ gates. In step \uppercase\expandafter{\romannumeral3}, 'uncomputing' means that repeat previous steps to restore to the initial state $|0\rangle$, so the gate complexity in step \uppercase\expandafter{\romannumeral3} is the sum of the gate complexity of step \uppercase\expandafter{\romannumeral1} and step \uppercase\expandafter{\romannumeral2}.

Then, the total query complexity and gate complexity of $\mathcal{P}_{\boldsymbol{A}}$ are summarized in the following.

\begin{lemma} \label{complexity-oracle}
    Assume that $\mathcal{O}_{\boldsymbol{A}}$ returns the element of $\boldsymbol{A}$ with $s$ bits of accuracy. Then oracle $\mathcal{P}_{\boldsymbol{A}}$ can be performed using $O(1)$ queries to oracle $\mathcal{P}_A$ with gate complexity $O(\log (MN) + s^2)$.
\end{lemma}

To apply the HHL algorithm \cite{HHL} or the LCU method \cite{LCU} to $\boldsymbol{A}\boldsymbol{x} = \boldsymbol{b}$, we can use a particular constant $c$, from Lemma \ref{norm-A}, such that $c \geq 2$ to rewrite $\boldsymbol{A}\boldsymbol{x} = \boldsymbol{b}$ as $(\boldsymbol{A}/c)(\boldsymbol{x}c) = \boldsymbol{b}$. Obviously, the corresponding state of the solution of the rewritten linear system is still $|\boldsymbol{x}\rangle$, and the queries to oracle $\mathcal{P}_{\boldsymbol{A}/c}$ are equal to the queries to oracle $\mathcal{P}_{\boldsymbol{A}}$.

Next, consider oracle $\mathcal{P}_{\boldsymbol{b}}$ that generates state $|\boldsymbol{b}\rangle = \sum_i \boldsymbol{b}_i |i\rangle / \| \sum_i \boldsymbol{b}_i |i\rangle \|$, where $\boldsymbol{b}_i$ is the $i$-th element of $\boldsymbol{b}$. That is, consider oracle $\mathcal{P}_{\boldsymbol{b}}$ such that $\mathcal{P}_{\boldsymbol{b}} |0^{nm}\rangle = |\boldsymbol{b}\rangle$. From the definition of $\boldsymbol{b}$, $\mathcal{P}_{\boldsymbol{b}}$ can be represented as
\begin{equation}
\mathcal{P}_{\boldsymbol{b}} = H^{\otimes m} \otimes \mathcal{P}_b,
\end{equation}
which leads to that the query complexity of $\mathcal{P}_{\boldsymbol{b}}$ is equal to oracle $\mathcal{P}_{b}$'s, where $H$ is Hadamard operator. Furthermore, the gate complexity of $\mathcal{P}_{\boldsymbol{b}}$ is $O(\log (MN))$. Therefore, the conditions for solving QLSP $(\boldsymbol{A}/c)(\boldsymbol{x}c) = \boldsymbol{b}$ are satisfied.

\subsection{Total complexity}
In this subsection, we consider the total query and gate complexity of this algorithm. Except the measurement step, we need to analyze both query and gate complexity from Step 1 to Step 4, and derive the following theorem.

\begin{theorem} \label{th-complexity}
    Consider the quantum algorithm described in Section \ref{algorithm}. Assuming that Step 1 outputs a state $|\tilde{\boldsymbol{x}}\rangle$ such that $\| |\boldsymbol{x}\rangle - |\tilde{\boldsymbol{x}}\rangle \| \leq \epsilon'$, where $0 \leq \epsilon' \leq 1/2$. Assuming that unitary $U$ obtained by block-encoding technique in Step 4 is implemented with $O(L)$ primitive gates. Then, to implement the proposed algorithm, we need
    \begin{equation}
    O\left(d\kappa'^2 \log^2 \left(\frac{d\kappa'}{\epsilon'} \right) \right)
    \end{equation}
    queries to oracle $P_A$ and
    \begin{equation}
        O\left(\kappa' \log \left( \frac{d\kappa'}{\epsilon'} \right) \right)
    \end{equation}
    uses of $P_b$ with gate complexity
    \begin{equation}
    O\left( d\kappa'^2 \log^2 \left(\frac{d\kappa'}{\epsilon'}\right) \left[ \log (MN) + \log^\frac{5}{2} \left(\frac{d\kappa'}{\epsilon'}\right)\right] + L \right),
    \end{equation}
    where $\kappa' = 1/(1-\|A-I\|)$ and $d$ is the sparsity of matrix $A$.
\end{theorem}

\begin{proof}
    Combined with Lemma \ref{complexity-oracle}, Corollary \ref{condition-num} and Theorem \ref{HHL-LCU}, we can get that Step 1 needs $O\left(d\kappa'^2 \log^2 \left(\frac{d\kappa'}{\epsilon'} \right) \right)$ queries to oracle $ P_A$ and $ O\left(\kappa' \log \left( \frac{d\kappa'}{\epsilon'} \right) \right)$ uses of $P_b $ with gate complexity
    $ O\left( d\kappa'^2 \log^2 \left( \frac{d\kappa'}{\epsilon'} \right) \left[ \log (MN) + \log^\frac{5}{2} \left( \frac{d\kappa'}{\epsilon'} \right)\right] \right) $. Step 2 is only gate complexity $O(1)$. It is obviously that the gate complexity of Step 3 is $O(m) = O(\log M)$. The gate complexity of Step 4 is $O(L)$. Thus, the theorem follows.
\end{proof}

In order to increase the success probability, we always use amplitude amplification technique, which needs to repeat Step 1 to Step 4 of the algorithm $O(1/\sqrt{\tilde{p}})$ times \cite{amp}, before the measurement.

\section{Conclusion}
In this paper, for the problem of computing $ \log(A) b$, we
proposed a quantum algorithm to compute the state $|f\rangle =
\log(A) |b\rangle / \| \log(A) |b\rangle \|$. The method is
motivated by using quadrature rules on the integral representation
of $\log(A)$, one of the classical methods of computing $\log(A)$.
In this method, we choose the Gauss-Legendre quadrature rule, while
the trapezoidal rule was applied in \cite{tosu20}. We use the fact
that the $M$-point Gauss Legendre quadrature is just the Pad\'{e}
approximation to the matrix logarithm \cite{gl-pade}. And this is
used for error estimation. The complex function must be analytical
on the disk with center $0$ in \cite{tosu20}, while we do not have
this restriction. When we prepare the manuscript, we notice the work
\cite{tosu21}. Tow popular techniques, LCU method and
block-encoding, can be used as subroutines. The block-encoding
framework can also be applied to matrix logarithm. Precisely
speaking, using $A$'s block-encoding $U_A$, we construct the
block-encoding of $A_j=\tau_j A + (1-\tau_j)I$. Based on the
block-encoding of matrix inverse \cite{GSLW19,tosu21}, we can
implement the block-encoding of $A_j^{-1}$, and then of
$h_M(A)=\Sigma_j \omega_j A_j^{-1}$.

\end{document}